\documentclass[10pt,twoside]{article}
\usepackage{amssymb}
\usepackage{amssymb}
\usepackage{amsmath}
\usepackage{amsfonts}
\newtheorem{theorem}{Theorem}[section]
\newtheorem{definition}{Definition}[section]
\newtheorem{lemma}{Lemma}[section]

\newtheorem{remark}{Remark}[section]


\begin{document}
\title{
\vspace{0.5in} {\bf\Large  Invariant manifolds for random and
stochastic  partial differential equations}}
\author{{\bf\large Tom\'as Caraballo}\footnote{The author acknowledges the support
Ministerio de Educaci\'on y Ciencia (Spain) BFM2002-03068 and
MTM2005-01412}\hspace{2mm}
\vspace{1mm}\\
{\it\small Dpto. Ecuaciones Diferenciales y An\'alisis
Num\'erico},\\ {\it\small Universidad de Sevilla}, {\it\small
Apdo. de Correos 1160,
41080-Sevilla, Spain}\\
{\it\small e-mail: caraball@us.es}\vspace{1mm}\\
{\bf\large Jinqiao Duan}\footnote{The author acknowledges the
support NSF 0620539}\hspace{2mm}
\vspace{1mm}\\
{\it\small Department of Applied Mathematics},\\ {\it\small Illinois Institute of Technology}, {\it\small Chicago, IL 60616, USA}\\
{\it\small e-mail: duan@iit.edu}\vspace{1mm}\\
{\bf\large Kening Lu}\footnote{The author acknowledges the support
NSF DMS 0200961 and  NSF DMS
0401708}\vspace{1mm}\vspace{1mm}\\
{\it\small Department of Mathematics},\\ {\it\small Brigham Young University}, {\it\small Provo, Utah 84602, USA}\\
{\it\small e-mail: klu@math.byu.edu}\vspace{1mm}\\
{\bf\large Bj{\"o}rn Schmalfu{\ss}}\footnote{The author
acknowledges the support DFG17355596}\vspace{1mm}\\
{\it\small Institut f{\"u}r Mathematik},\\ {\it\small
Universit{\"a}t Paderborn}, {\it\small 33098 Paderborn,
  Germany}\\
{\it\small e-mail: schmalfuss@math.upb.de}\vspace{1mm}}

\maketitle

\newpage

\begin{center}
{\bf\small Abstract}

\vspace{3mm} \hspace{.05in}\parbox{4.5in} {{\small Random
invariant manifolds  are   geometric objects useful for
understanding complex dynamics under stochastic influences. Under
a nonuniform hyperbolicity or a nonuniform exponential dichotomy
condition,  the existence of random pseudo-stable and
pseudo-unstable manifolds for a class of \emph{random} partial
differential equations  and \emph{stochastic} partial differential
equations is shown. Unlike the invariant manifold theory for
stochastic \emph{ordinary} differential equations, random norms
are not used. The result is then applied to a nonlinear stochastic
partial differential equation with linear multiplicative noise. }}
\end{center}


   \newcommand{\eps}{\varepsilon}
   \newcommand{\Wsob}{\smash{{\stackrel{\circ}{W}}}_2^1(D)}
   \newcommand{\EX}{{\Bbb{E}}}
   \newcommand{\PX}{{\Bbb{P}}}
   \parindent0mm

\renewcommand{\k}{\kappa}
\newcommand{\e}{\epsilon}
\renewcommand{\a}{\alpha}
\renewcommand{\b}{\beta}
\newcommand{\om}{\omega}
\newcommand{\Om}{\Omega}
\newcommand{\D}{\Delta}
\newcommand{\p}{\partial}
\newcommand{\de}{\delta}
\renewcommand{\phi}{\varphi}
\newcommand{\N}{{\mathbb N}}
\newcommand{\R}{{\mathbb R}}
\newcommand{\cF}{{\cal F}}

\setcounter{secnumdepth}{5} \setcounter{tocdepth}{5}

\makeatletter
    \newcommand\figcaption{\def\@captype{figure}\caption}
    \newcommand\tabcaption{\def\@captype{table}\caption}
\makeatother



\section{\bf Introduction}

Invariant structures in state spaces are   essential for
describing and understanding dynamical behavior of nonlinear
random systems. For random dynamical systems, these invariant
structures are usually random geometric objects.  Stable,
unstable,  center,    and inertial manifolds, as special random
invariant structures, have been considered in the investigation of
stochastic partial differential equations or stochastic
evolutionary equations in infinite dimensional spaces
\cite{CarLanRob01, DLS1, DLS2, Mohzhazha03, BenFla95, GirChu95,
DaPDeb96, WangDuan}. More detailed historical account of this
subject may be found in \cite{CarLanRob01, DLS1}.

\medskip

In this paper, we are concerned with invariant stable or unstable
manifolds for infinite dimensional random dynamical systems,
especially those systems generated by {\em stochastic} or {\em
random} partial differential equations (SPDEs or RPDEs), under
some weak conditions. Our approach for establishing invariant
manifolds for infinite dimensional {\em random} dynamical systems
is based on a nonuniform  exponential dichotomy, also called
nonuniform pseudo-hyperbolicity, for the linearized random
dynamical systems. When a multiplicative ergodic theorem  (MET)
holds \cite{Rue82, LianLu}, nonuniform pseudo-hyperbolicity also
holds. Moreover, unlike the invariant manifolds theory for finite
dimensional random dynamical systems \cite{Wanner, Arn98}, we make
no use of random norms. To be more precise, the structure of our
analysis is the following. Before proving the existence of
invariant manifolds for a nonlinear (random or stochastic) partial
differential equation (PDE),
 we analyse the linear system as a first
approximation. We prove that the fundamental solutions of our
linear PDE generates a random dynamical system that is linear and
compact (for every positive time $t$). The partial differential
operator generating this equation is supposed to be uniformly
elliptic and random. The long-time behaviour of this linear random
dynamical system is analysed under a nonuniform
pseudo-hyperbolicity  condition, which also implies an exponential
dichotomy result. We then use a cut-off procedure to obtain the
existence of local (pseudo) invariant stable and unstable
manifolds for nonlinear random systems by using the
Lyapunov-Perron technique.

\medskip

The paper is organized as follows.
  In Section \ref{s2}, we recall some  basic concepts for
random dynamical systems. In Section \ref{ED}, we discuss
multiplicative ergodic theorems and exponential dichotomies for
linear cocycles. We prove that when a multiplicative ergodic
theorem (MET) holds in an infinite dimensional Hilbert space, a
\emph{nonuniform} exponential dichotomy (i.e., nonuniform
pseudo-hyperbolicity)  also holds (Theorem \ref{t4}) in the same
Hilbert space. Furthermore, we obtain sufficient conditions under
which
  a stochastic partial differential equation generates a continuous random
dynamical system (Theorem \ref{cocycle}). We then prove
pseudo-stable and pseudo-unstable manifold theorems  for random
and stochastic partial differential equations (Theorems
\ref{thm3.1} and \ref{thm3.2}), under nonuniform
pseudo-hyperbolicity (see Definition \ref{hyperbolic}), in Section
\ref{IM}. Finally, in Section \ref{appl}, we demonstrate our
invariant manifold theorem for an example of  stochastic partial
differential equations.

\section{\bf Random dynamical systems}\label{s2}

We now recall some basic concepts in random dynamical systems.
 First we introduce  an appropriate model for a noise. Such a
model is given by a metric dynamical system defined by a quadrupel
$(\Omega,\mathcal{F},\mathbb{P},\theta)$, where
$(\Omega,\mathcal{F},\mathbb{P})$ is a probability space and
$\theta$ is   a measurable flow with time set $\mathbb{T}$ being
$\mathbb{R}$ or $\mathbb{Z}$:
\begin{align*}
    \theta: (\mathbb{T}\times
    \Omega,\mathcal{B}(\mathbb{T})\otimes\mathcal{F})\to
    (\Omega,\mathcal{F}).
\end{align*}
For the partial mappings $\theta(t,\cdot)$ we use the symbol
$\theta_t$. We then have
\begin{equation*}
    \theta_t\circ
    \theta_\tau=:\theta_t\theta_\tau=\theta_{t+\tau}\quad\text{ for
    }t,\,\tau\in\mathbb{T},\qquad \theta_0={\rm id}_\Omega.
\end{equation*}
The measure $\mathbb{P}$ is taken to be ergodic with respect
to the {\em shift} operators $\theta_t$; see \cite{Boxler}.\\
 The standard example
for a metric dynamical system is induced by the Brownian motion.
Let $V$ be a separable Hilbert space and let $C_0(\mathbb{R},V)$
be the set of continuous functions on $\mathbb{R}$ with values in
$U$ which are zero at zero equipped with the compact open
topology. We denote by $\mathcal{F}$ the associated
Borel-$\sigma$-algebra. Let $\mathbb{P}$ be the Wiener measure on
$\mathcal{F}$ which is given by the distribution of a two-sided
Wiener process with trajectories in $C_0(\mathbb{R},V)$. For the
definition of a two-sided Wiener process see Arnold \cite{Arn98}
Page 547. The flow $\theta$ is given by the Wiener shifts
\begin{equation*}
    \theta_t\omega(\cdot)=\omega(\cdot+t)-\omega(t),\quad
    t\in\mathbb{R},\qquad \omega\in\Omega=C_0(\mathbb{R},V).
\end{equation*}
In this case the measure $\mathbb{P}$ is {\em ergodic} with
respect to the flow $\theta$.

\medskip

For some Polish space (complete separable metric space) $H$ a
random dynamical system is given by a mapping
\begin{equation*}
    \phi:
    (\mathbb{T}^+\times\Omega\times H,\mathcal{B}(\mathbb{T}^+)\otimes
    \mathcal{F}\otimes \mathcal{B}(H))\to (H,\mathcal{B}(H))
\end{equation*}
which has the {\em cocycle} property:
\begin{align}\label{eq1}
\begin{split}
    \phi(t+\tau,\omega,x)&=\phi(t,\theta_\tau\omega,\phi(\tau,\omega,x)),\qquad
    t,\,\tau\in\mathbb{T}^+,\quad \omega\in\Omega,\\
    \phi(0,\omega,x)&=x.
    \end{split}
\end{align}
Cocycles are generalizations of semigroups reflecting some
non-autonomous dynamics.

Suppose that for some flow $\theta$ the differential equation
\begin{equation*}
    u^\prime=f(\theta_t\omega,u),\quad u(0)=x\in H
\end{equation*}
has a unique solution on any interval $[0,T]$ for $T>0$. Then the
solution mapping $(t,\omega,x)\to\phi(t,\omega,x)$ defines a
cocycle. If this operator depends measurably on its variables then
$\phi$ defines
a random dynamical system.\\

In what follows we have to transform one random dynamical system
into another. To do this we need the following lemma.
\begin{lemma}\label{lm1}
Consider  the mapping
\begin{align*}
    &T:\Omega\times H\to H,
\end{align*}
and assume that $T(\omega,\cdot)$ is a homeomorphism for any
$\omega\in\Omega$, and $T(\cdot,x),\,T(\cdot,x)^{-1}$ are
measurable for any $x\in H$. If $\phi$ is a continuous random
dynamical system, then so is $\psi$ defined by
\begin{equation*}
    \psi(t,\omega,x):=T(\theta_t\omega,\phi(t,\omega,T(\omega,x)^{-1}).
\end{equation*}
\end{lemma}
The proof is straightforward. We note that, by the assumptions of
the lemma, the mappings $T$ and $T^{-1}$ are measurable from
$\Omega\times H$ to $H$, see Castaing and Valadier \cite{CasVal77}
Lemma III.14.

\medskip

For our purpose, a class of random variables will be crucial. A
random variable
\begin{equation}\label{eqA5}
X:(\Omega,\mathcal{F})\to
(\mathbb{R}^+\setminus\{0\},\mathcal{B}(\mathbb{R}^+\setminus\{0\}))
\end{equation}
  is called {\em tempered} if

\begin{equation*}
    \lim_{t\to\infty}\frac{\log^+ X(\theta_t\omega)}{t}=0
\end{equation*}
for $\omega$ contained in a $\{\theta_t\}_{t\in\mathbb{R}}$
invariant set of full measure. Such a random variable $X$ is
called {\em tempered from below}  if $X^{-1}$ is tempered. We note
that, in the case of ergodicity, a random variable defined in
(\ref{eqA5}) is either tempered  or alternatively there exists a
$\{\theta_t\}_{t\in\mathbb{T}}$ invariant set $\tilde \Omega$ of
full measure such that
\begin{equation*}
    \limsup_{t\to\pm\infty}\frac{\log^+
    X(\theta_t\omega)}{t}=+\infty,\quad\omega\in\tilde\Omega.
\end{equation*}
A random variable is tempered if and only if there exists a
postive constant $\Lambda$ and a  positive random variable
$C_\Lambda(\omega)$ such that
\begin{equation}\label{eq8}
    X(\theta_t\omega)\le C_\Lambda(\omega)e^{\Lambda\,t}\quad
    \text{for}\quad t\in  \mathbb{T}
\end{equation}
for $\omega$ in some $\{\theta_t\}_{t\in\mathbb{T}}$ invariant set
$\tilde \Omega$ of full measure.\\


We need the following definitions and conclusions about the
measurability of linear operators.

Let $H_1,\,H_2$ be separable Hilbert spaces. A mapping $\omega\to
B(\omega)\in L(H_1,H_2)$ is called {\em strongly measurable} if
$\omega\to B(\omega)h$ is a random variable on $H_2$ for every
$h\in H_1$.

\begin{lemma}\label{l4}
Let $H_1,\,\,H_2,\,H_3$ be three separable Banach spaces.
Let $B$ be a strongly measurable operator in $L(H_1,H_2)$, and let $C$ be a strongly measurable operator in $L(H_2,H_3)$. Then\\
(i) $B:\Omega\times H_1\to B(\omega)h\in H_2$ is  measurable. \\
(ii) $C\circ B$ is strongly measurable in $L(H_1,H_3)$\\
(iii) $\omega\to \|B(\omega)\|_{L(H_1,H_2)}$ is measurable.\\
(iv) Let $\dot{H}_1$ a dense set in $H_1$ and suppose that
$\omega\to B(\omega)h$ is measurable for $h\in \dot{H}_1$. Then
$B$ is strongly measurable.
\end{lemma}
\begin{proof}
(i) Follows from Castaing and Valadier \cite{CasVal77} Lemma
III.14. and (ii) is a consequence of (i). (iii) follows because
the unit ball in $H_1$ contains a dense countable set and for (iv)
we note that $B(\cdot)h,\,h\in H_1$, is the pointwise limit for
some sequence $(B(\cdot)h_n),\,h_n\in \dot{H}_1$.  \hfill$\square$
\end{proof}


\section{\bf Multiplicative ergodic theorem and exponential dichotomy}  \label{ED}

In this section we introduce random dynamical systems $U$
consisting of linear continuous operators $U(t,\omega)\in L(H,H)$.
In particular, we now study linear random dynamical systems
generated by random linear evolution equations
\begin{equation}\label{eq2}
    \frac{du}{dt}+A(\theta_t\omega)u=0,\qquad u(0)=x\in H.
\end{equation}

To describe the properties of the operator $A$ let
$(H,(\cdot,\cdot),\|\cdot\|),\,(H_1,(\cdot,\cdot)_1,\|\cdot\|_1)$
be two separable Hilbert spaces where $H_1$ is densely and compact
embedded into $H$. We assume that $A$ is given by linear operators
$A(\omega)\in L(H_1,H)$ such that $\omega\to A(\omega)$ is
strongly measurable. In addition, $-A(\omega)$ are generators of
an analytic $C_0$--semigroups on $H$ denoted by $e^{-\tau
A(\omega)},\,\tau\ge 0$, and the function
 $t\to A(\theta_t\omega)$ is H{\"o}lder continuous with values in
 $L(H_1,H)$. Namely, the function
\begin{align*}
\mathbb{R}\ni t\to A(\theta_t\omega)
\end{align*}
is  in $C^\rho(\mathbb{R},L(H_1,H))$ for $  \rho\in (0,1)$. For
the definition of this space see Amann \cite{Ama95} Page 40f. We
also assume that there exists a random variable $k_1(\omega)\ge 0$
so that the resolvent set of $-(k_1(\omega)+A(\omega))$ denoted by
$\rho(-(k_1(\omega)+A(\omega)))$ contains $\mathbb{R}^+$, and the
mapping $t\to k_1(\theta_t\omega)$ is supposed to be H{\"o}lder
continuous. We define
$A_\omega(t):=A(\theta_t\omega),\,k_{1,\omega}(t):=k_1(\theta_t\omega)$
for $\omega\in\Omega$. According to the above properties,
$A_\omega$ generates a {\em fundamental solution} $U_\omega$ (or a
parabolic evolution operator), see Amann \cite{Ama95}. For our
application we need the following parts of the definition of a
fundamental solution. Let $J=[0,T],\,T>0$ or $\mathbb{R}^+$.
\begin{align}
    \label{eq3} U_\omega&\in C(J_\Delta,L_s(H)),\quad
    J_\Delta=\{(t,s)\in J^2,t\ge s\}.\\
    \label{eq4}
    &t\to U_\omega(t,s)x,\,x\in H,\,t\ge s\quad \text{solves}\\
    &\qquad \frac{du}{dt}+A_\omega(t)u=0,\; u(s)=x,\quad
\text{where}\; U_\omega(\cdot,s)\in  C^1(J\cap (s,\infty);L(H))
    .\nonumber \\
    \label{eq5} &U_\omega(t,t)={\rm id},\qquad
    U_\omega(t,s)=U_\omega(t,\tau)\circ U_\omega(\tau,s),\quad
    T>t\ge \tau\ge s.\\
    &\label{eq6} \sup_{T\ge t\ge
    s\ge 0}(t-s)\|A_\omega(t)U_\omega(t,s)\|<\infty.
\end{align}
$L_s(H)$ denotes the strong convergence on the set of continuous
linear operators $L(H)$ on $H$. For the operator norm
$L(H_1,H),\,L(H)$ we simply write $\|\cdot\|$. In addition, we
have
\begin{equation}\label{eq6a}
\|U_\omega(t,s)\|\le C_\omega e^{\mu_\omega(t-s)}
\end{equation}
for appropriate constants $C_\omega,\,\mu_\omega$, see Amann
\cite{Ama95} Theorem II.4.4.1.\\

We consider the following simple transform
\begin{equation}\label{eq0.1}
    U_{k_1,\omega}(t,s):=e^{-\int_0^tk_{1,\omega}(\tau)d\tau}U_\omega(t,s)e^{\int_0^s k_{1,\omega}(\tau)d\tau}.
\end{equation}
These operators are fundamental solutions of an equation generated
by
\begin{equation*}
    A_{k_1,\omega}(t)=k_{1,\omega}(t) {\rm id}+
    A_\omega(t).
\end{equation*}

We have $k_{1,\omega}(t)\le K_{1,\omega,T}$  on every interval
$[0,T]$. Then we can introduce the operator

\begin{equation*}
    A_{K_1,\omega,T}(t)=K_{1,\omega,T} {\rm id}+
    A_\omega(t).
\end{equation*}

For such a $K_{1,\omega,T}$, the condition (II.4.2.1) in Amann
\cite{Ama95}, Page 55, is satisfied on $[0,T]$. This gives us the
existence of a  unique fundamental solution with generator
$A_{K_1,\omega,T}$  and hence with generator $A_\omega$; see
\cite{Ama95}, Corollary II.4.4.2. In particular, for any $T>0$ we
have some $M_{T,\omega}$ such that
\begin{equation*}
    M^{-1}_{T,\omega}\|x\|_1\le \|A_{K_1,\omega,T}(t)x\|\le M_{T,\omega}\|x\|_1.
\end{equation*}
We then can conclude by (\ref{eq6})
\begin{equation*}
    \sup_{T\ge t>
    s\ge 0}(t-s)\|U_{K_{1,\omega,T}}(t,s)\|_1\le
    M_{T,\omega}\sup_{T\ge t>
    s\ge 0}(t-s)\|A_{K_{1,\omega,T}}(t)U_{K_{1,\omega,T}}(t,s)\|<\infty
\end{equation*}
such that $U_{K_{1,\omega,T}}(t,s),\,s<t\in J$ and hence
$U_\omega(t,s)$ for $t>s$ are compact linear operators by the
compact embedding
$H_1\subset H$. For $t=0$ case, see \cite{Ama95}.\\

Our intention is  now to derive from the fundamental solution a
random dynamical system. We  set
\begin{equation*}
    U(t,\omega):=U_\omega(t,0).
\end{equation*}
By $A_\omega(t)=A_{\theta_s\omega}(t-s),\,t\ge s$ the cocycle
property follows directly from (\ref{eq5})
\begin{equation}\label{eq7}
    U(t+\tau,\omega)=U(t,\theta_\tau\omega)\circ U(\tau,\omega)
\end{equation}

Replacing $A_\omega$ by the operator given by (\ref{eq0.1}) we can
assume that the resolvent set of $-A(\omega)$ contains
$\mathbb{R}^+$.

Considering the Yoshida approximations

\begin{equation*}
    A^\eps(\omega)=A(\omega)({\rm id}+\eps A(\omega))^{-1}\in
    L(H,H).
\end{equation*}

By our assumptions on the resolvent set these operators are
defined for $\eps>0$.

Then the solution of the equation

\begin{equation*}
    \frac{du}{dt}+A^\eps(\theta_t\omega)u=0,\quad u(0)=x
\end{equation*}

can be constructed by Picard iteration so that the associated
fundamental solution $U^\eps$ forms a random dynamical systems if
$A^\eps$ is strongly measurable. In particular, we note that from
Amann (II.6.1.9) follows that

\begin{equation*}
    t\to \|A^\eps(\theta_t\omega)\|
\end{equation*}

is H{\"o}lder continuous, hence locally integrable.\\

We have to prove that the Yoshida approximations are strongly
measurable. Indeed for $h\in H$ the operator
$(\lambda+A(\omega))^{-1}$ exists for every $\lambda>0$ as an
operator in $L(H,H_1)$. By Skorochod \cite{Sko84} Chapter II.6.3
the random variable $(\lambda+A(\omega))^{-1}h$ is measurable with
respect to $\mathcal{B}(H)$. But we have

\begin{equation*}
    \mathcal{B}(H)\cap H_1=\mathcal{B}(H_1),
\end{equation*}

see  Vishik and Fursikov\cite{VisFur88} Chapter II.2. which gives
the strong measurability of $A^\eps$.

Then by the convergence of the Yoshida approximations we have the
pointwise limit

\begin{equation*}
    \lim_{\eps\to 0}U^\eps(t,\omega)x=U(t,\omega)x
\end{equation*}

for every $x\in H$ (see Amann Theorem II.6.2.4) which shows that
$U$ is a random dynamical system. In particular, it holds, By
(\ref{eq3}), that the mapping $t\to U_\omega(t,0)x$ is continuous
for $t\ge 0$. Hence due  to Castaing and Valadier \cite{CasVal77}
Lemma III.14.
\begin{equation*}
    (t,\omega)\to U_\omega(t,0)x
\end{equation*}
is measurable.  Similarly, since $(t,\omega)\to U_\omega(t,0)x$ is
measurable for fixed $x\in H$, and the mapping $x\to
U_\omega(t,\omega)x$ is continuous, we have that
\begin{equation*}
    (t,\omega,x)\to U_\omega(t,0)x
\end{equation*}
is measurable. Together with (\ref{eq7}), $U$ defines a continuous
random dynamical system. If we consider the original random
dynamical system
by the inverse transform to (\ref{eq0.1}) we can conclude that $A$ generates a random dynamical system.\\

Summarizing the above discussions, we have the following theorem
on linear cocycles.

\begin{theorem}  (\textbf{Generation of linear cocycle})  \label{t1}  \\
Let $A(\omega)\in L(H_1,H)$ be generators of  analytic
$C_0$--semigroups on $H$. The separable Hilbert space $H_1$ is
compactly and densely embedded  in the separable Hilbert  space
$H$. In addition, we assume that $t\to A(\theta_t\omega)$ is
H{\"o}lder continuous in $L(H_1,H)$ and $\Omega\ni\omega\to
A(\omega)\in L(H_1,H)$ is strongly measurable, and the resolvent
set of $-A(\omega)$ contains the interval $[k_1(\omega),\infty)$
where $k_1\ge 0$ is a random variable.  Then (\ref{eq2}) generates
a random dynamical system of compact linear operators on $H$.
\end{theorem}

We consider the following example. Let $A$ be the following linear
differential operator over a bounded domain
$\mathcal{O}\in\mathbb{R}^d$ with $C^\infty$-smooth boundary
$\partial \mathcal{O}$,
\begin{equation}\label{opA}
    A(x,\omega,D)u=\sum_{|\gamma|,|\delta|\le
    m}(-1)^{|\gamma|}D^\gamma(a_{\gamma,\delta}(x,\omega)D^\delta)u.
\end{equation}

We suppose that $a_{\gamma,\delta}$ forms a  stochastic process

\begin{equation*}
    (t,\omega)\to a_{\gamma,\delta}(\theta_t\omega,\cdot)\in
    C^m({\bar {\mathcal{O}}})
\end{equation*}

which has H{\"o}lder continuous path. The principal part of  $A$

\begin{equation*}
    A_0(x,\omega,D)u=\sum_{|\gamma|,|\delta|=
    m}(-1)^{|\gamma|}D^\gamma(a_{\gamma,\delta}(x,\omega)D^\delta)u.
\end{equation*}

is supposed to be uniformly elliptic

\begin{equation*}
    \sum_{|\gamma|,|\delta|=m}a_{\gamma,\delta}(x,\omega)z_\gamma
    z_\delta\ge 2k_0(\omega)|z|^m,\quad z=(\cdots, z_\gamma,\cdots)
\end{equation*}

where the vector $z$ is indexed by the multi--index $\gamma$. The
random variable $k_0(\omega)\in (0,\infty)$ is supposed to be
independent of $x\in {\bar {\mathcal{O}}}$. We also assume that
$t\to k_0^{-1}(\theta_t\omega)$ is H{\"o}lder continuous. The
differential operator will be augmented by boundary conditions

\begin{equation}\label{opA2}
    u|_{\partial \mathcal{O}}=\frac{\partial u}{\partial n}|_{\partial
    \mathcal{O}}=\cdots = \frac{\partial u^{m-1}}{\partial n^{m-1}}|_{\partial
    \mathcal{O}}
\end{equation}

and $n$ denotes the outer normal.  We set
\begin{equation*}
    H=L^2(\mathcal{O}),\quad V= H_{0}^m(\mathcal{O}),\quad H_1=V\cap H^{2m}(\mathcal{O})
\end{equation*}
where $H_0^m$ and $H^{2m}$ are standard Sobolev spaces. A more
specific example is $A=\Delta$ (Laplace operator), under the zero
Dirichlet boundary condition.
 \\

We introduce the continuous bilinear form on $V$:

\begin{equation*}
    b_\omega(u,v)=\sum_{\gamma,\delta\le m}(a_{\gamma,\delta}(\omega) D^\delta
    u,D^\gamma v)+k_1(\omega)(u,v).
\end{equation*}

satisfying the Lax--Milgram condition

\begin{equation*}
b_\omega(u,v)\ge k_0(\omega)\|u\|_V^2.
\end{equation*}

Then $A(x,\omega,D)$ generates an analytic $C_0$--semigroup in $H$
with generator denoted by $A(\omega)$ and $D(A(\omega))= H_1$ for
every $\omega\in\Omega$. We note that

\begin{equation*}
   \omega\to  A(\cdot,\omega,D)h\in H,\qquad h\in C_0^\infty(\mathcal{O})
\end{equation*}

is measurable, so that $\omega\to A(\omega)h,\quad h\in H$, is
measurable. In addition, by the remarks about the processes
$a_{\gamma,\delta}(\theta_t\omega,\cdot)$, the operators
$A(\theta_t\omega)$ are in  $L(H_1,H)$ so that we have to ensure
that  terms like

\begin{equation*}
    \sup_{h\in
    C_0^\infty(\bar{\mathcal{O}}),\|h\|_{1}=1}\|(A(\theta_t\omega)-A(\theta_s\omega))h\|^2
\end{equation*}

are H{\"o}lder continuous for $\omega\in\Omega$. Indeed, we have
for appropriate $\gamma,\,\delta$

\begin{align*}
    \int_\mathcal{O}&|D^\gamma(a_{\gamma,\delta}(\theta_s\omega,x)-a_{\gamma,\delta}(\theta_t\omega,x))D^\delta
    h(x)|^2dx\\
    &\le
    \sup_{x\in \bar{\mathcal{O}}}|D^\gamma
    a_{\gamma,\delta}(x,\theta_s\omega)-D^\gamma
    a_{\gamma,\delta}(x,\theta_t\omega)|^2\int_\mathcal{O}|D^\delta
    h(x)|^2dx\\&\le
    \|a_{\gamma,\delta}(\theta_s\omega)-a_{\gamma,\delta}(\theta_t\omega)\|_{C^m(\bar{\mathcal{O}})}^2
    \|h\|_{{W_2^{2m}(\bar{\mathcal{O}})}}^2.
    &
\end{align*}

We note that $k_1$ can be calculated by an interpolation argument.
Then, by the assumptions on $a_{\gamma,\delta}$ and $k_0$, it
follows that $t\to k_1(\theta_t\omega)$ is H{\"o}lder continuous.

In the following we describe the stability behavior of linear
random dynamical systems with an infinite dimensional state space.
To do this we formulate an infinite dimensional version of the
multiplicative ergodic theorem; see Ruelle \cite{Rue82}. A version
of this theorem for continuous time can be found in Mohammed et
al. \cite{Mohzhazha03}.
\begin{theorem}\label{t2}
Let $U$ be a linear random dynamical system of compact operators
for $t>0$ on $H$ satisfying the following integrability condition
\begin{equation*}
    \mathbb{E}\sup_{0\le t\le 1}\log^+\|U(t,\omega)\|+
    \mathbb{E}\sup_{0\le t\le
    1}\log^+\|U(1-t,\theta_t\omega)\|<\infty.
\end{equation*}
Then there exist finitely or infinitely many deterministic numbers
$\lambda_1>\lambda_2>\cdots$ (with $-\infty$ possible) and linear
spaces $H=V_1\supset V_2(\omega)\supset \cdots$, such that\\
(i)  Each linear space  $V_i(\omega)$ has a finite co-dimension
independent of $\omega$. \\
(ii) The following limits and invariance conditions hold:
\begin{align*}
   & \lim_{t\to\infty}\frac{1}{t}\log\|U(t,\omega)x\|=\lambda_i
    \qquad \text{for}\quad x\in V_i(\omega)\setminus
    V_{i+1}(\omega)\\
    &
    U(t,\omega)V_i(\omega)\subset V_i(\theta_t\omega)\quad
    \text{for}\quad t\ge 0
\end{align*}
for $t\ge 0$ and for all $\omega$ contained in a set
$\tilde\Omega$ of full measure such that
$\theta_t\tilde\Omega\subset\tilde\Omega,\,t\ge 0$.
\end{theorem}
The numbers $\lambda_1,\,\lambda_2,\cdots$ are called the {\em
Lyapunov exponents} associated to $U$. The set of these numbers
forms the
{\em Lyapunov spectrum}.\\

By the above theorem we can derive the following exponential
dichotomy condition for $U$; see Mohammed et al.
\cite{Mohzhazha03}.

\begin{theorem}\label{t3}
Suppose that the following exponential integrability condition is
satisfied
\begin{equation}\label{eq21}
    D(\omega):=\log^+\sup_{t_1,\,t_2\in[0,1]}\|U(t_1,\theta_{t_2}\omega)\|,\qquad \mathbb{E}D<\infty
\end{equation}
and suppose that $\lambda\in\mathbb{R}$ is not contained in the
Lyapunov spectrum. Then there exists a
$\{\theta_t\}_{t\in\mathbb{R}}$ invariant set $\tilde \Omega$ of
full measure such that for $\omega\in\tilde \Omega$ we have the
following properties: There exist linear spaces
$E^u(\omega),\,E^s(\omega)$ such that
\begin{equation*}
    H=E^u(\omega)\oplus E^s(\omega).
\end{equation*}
The space $E^u(\omega)$ has a finite dimension independent of
$\omega$;
\begin{align*}
    U(t,\omega)E^u(\omega)&=E^u(\theta_t\omega)\\
    U(t,\omega)E^s(\omega)&\subset E^s(\theta_t\omega)
\end{align*}
for $t\ge 0$. The restriction of $U(t,\omega)$ to $E^u(\omega)$ is
invertible. There exist measurable projections
$\Pi^u(\omega),\,\Pi^s(\omega)$ onto $E^u(\omega),\,E^s(\omega)$.
 In the case that $\lambda_1<\lambda$ we have
$E^u=\{0\}$.\\
Suppose that $\lambda_1>\lambda $ and let $\lambda_+$ be the
smallest Lyapunov exponent bigger than $\lambda$ and let
$\lambda_-$ be the biggest Lyapunov exponent smaller than
$\lambda$. Then we have for any $\eps>0$
\begin{align*}
    \|U(t,\omega)x\|&\ge\|x\|e^{\alpha t}\quad \text{for }
    x\in E^u(\omega),\quad
    t\ge\tau(\omega,\eps,x),\quad \alpha=\lambda_+-\eps\\
    \|U(t,\omega)x\|&\le\|x\|e^{\beta t}\quad \text{for }
    x\in E^s(\omega),\quad
    t\ge\tau(\omega,\eps,x),\quad \quad \beta=\lambda_-+\eps
\end{align*}
where $\eps$ is chosen so small that $\alpha>\beta$.
\end{theorem}
\begin{remark}\label{r1}
{\rm  i) The integrability condition of Theorem \ref{t3} ensures
the
integrability condition of Theorem \ref{t2}.\\
ii)
 The space $E^s(\omega)$ is given by $V_i(\omega)$ if
 $\lambda_-=\lambda_i$.\\
 iii) It follows directly from the invariance of the spaces
 $E^u(\omega)$ and $E^s(\omega)$, that
 \begin{align*}
    \Pi^u(\theta_t\omega)U(t,\omega)&=U(t,\omega)\Pi^u(\omega)\\
    \Pi^s(\theta_t\omega)U(t,\omega)&=U(t,\omega)\Pi^s(\omega)
\end{align*}
on a $\{\theta_t\}_{t\in\mathbb{R}}$-invariant set of full
measure.

}
 \end{remark}

We denote the restriction of $U(t,\omega)$ to
$E^u(\omega),\,E^s(\omega)$ by
$U_\lambda^u(t,\omega),\;U_\lambda^s(t,\omega)$:

\begin{equation*}
U_\lambda^{s/u}(t,\omega):E_\lambda^{s/u}(\omega)\to
E_\lambda^{s/u}(\theta_t\omega).
\end{equation*}

In the following we need the norm of these operators
$U_\lambda^{s/u}(t,\omega)$ which should be denoted by
$\|U_\lambda^{{s/u}}(t,\theta_s\omega)\|_{L(E^{s/u})(\theta_s\omega),E^{s/u}(\theta_{t+s}\omega))}$.
But to avoid these long expressions in the norm we simply write
$\|\cdot\|$ for the norm. From the context, this is not to be
confused with the norm in $H$.

\begin{lemma}\label{l2}
Suppose that the integrability condition of Theorem \ref{t3} is
satisfied. Then there exists a $\{\theta_t\}_{t\in\mathbb{R}}$
invariant set of full $\mathbb{P}$-measure and a constant
$H_\omega^\Lambda $ such that
\begin{equation*}
    \|U_\lambda^u(t,\omega)\|\le H_\omega^\Lambda e^{\Lambda
    t}\quad\text{for }t\ge 0,\,\omega\in\tilde\Omega
\end{equation*}
for a sufficiently large $\Lambda$.
\end{lemma}
\begin{proof}
1) We show that on a $\{\theta_t\}_{t\in\mathbb{R}}$ invariant set
of full measure
\begin{equation*}
\end{equation*}
By
\begin{equation*}
    \|U_\lambda^u(t,\omega)\|\le \|U(t,\omega)\|
\end{equation*}
we have that
$\mathbb{E}\log^u\|U_\lambda^u(1,\omega)\|=:\tilde\Lambda<\infty$.
By Kingman's theorem (see Ruelle \cite{Rue82}) there exists a set
of measure one such that for any $\omega$ in this set we have that
\begin{equation}\label{x1}
    \lim_{i\to\infty}\frac{1}{i}\log
    \|U_\lambda^u(i,\omega)\|=
    \tilde\Lambda.
\end{equation}
Hence
\begin{equation*}
    \Omega_n^1:=\{\omega\in\Omega:\limsup_{i\to\infty}
    \frac{1}{i}\log \|U_\lambda^u(i,\theta_n\omega)\|\le
    \tilde\Lambda\}\in\mathcal{F}\\
\end{equation*}

and set

\begin{equation*}
    \Omega^1:=\bigcap_{n\in\mathbb{Z}}\Omega_n^1.
\end{equation*}

This set is $\{\theta_t\}_{t\in\mathbb{Z}}$--invariant and has
probability one. Let $\Omega^2$ be the $\{\theta_t\}_{t\in
\mathbb{Z}}$--invariant set so that

\begin{equation*}
    \lim_{i\to\pm\infty}\frac{1}{i}D(\theta_{i+n}\omega)=0,\qquad
    n\in\mathbb{N}.
\end{equation*}

By the integrability condition (\ref{eq21}) and by

\begin{align*}
 D(\theta_t\omega)\le D(\theta_{1+[t]}\omega)+D(\theta_{[t]}\omega)
\end{align*}
we have that  $\Omega^2$ has full measure

\begin{equation*}
    \lim_{t\to\pm\infty}\frac{1}{t} D(\theta_t\omega)=0.
\end{equation*}

such that $\{\theta_t\}_{t\in\mathbb{R}}$--invariant. \\

2) Since $\Omega^1\cap\Omega^2$ is $\{\theta_t\}_{t\in\mathbb{Z}}$
invariant we can restrict
ourselves to the case that  $s\in (-1,0)$ for the invariance with respect to continuous time.\\
We have for $\omega\in\Omega^1\cap\Omega^2$ that
\begin{align}\label{eqn1}
\begin{split}
   \log\|U_\lambda^u(t,\theta_s\omega)\|\le &
    \log^+\|U_\lambda^u(s+t-[s+t],\theta_{[s+t]}\omega)\|\\
    +
    &\log\|U_\lambda^u([s+t],\omega)\|+\log^+\|U_\lambda^u(-s,\theta_s\omega)\|\\
    \le&
    D(\theta_{[s+t]}\omega)+\log\|U_\lambda^u([s+t],\omega)\|.
 \end{split}
\end{align}
(Note $s\le 0$). Thus
$\limsup_{t\to\infty}\log\|U_\lambda^u(t,\theta_s\omega)\|/t\le
\tilde \Lambda$. The same is true if we replace $\omega$ by
$\theta_{n}\omega,\,n\in\mathbb{Z}$. Hence
$\theta_s\omega\in\Omega^1$ and hence
$\theta_s\omega\in\Omega^1\cap \Omega^2$. On the other hand, for
$s=0$ we obtain the conclusion.                 \hfill$\square$
\end{proof}

The following lemma states that one can restrict a metric
dynamical system to a smaller invariant set of full measure.

\begin{lemma}\label{l3}
Let $\tilde\Omega$ be defined in Lemma \ref{l2} and let
$\mathcal{F}_{\tilde\Omega}$ be the trace $\sigma$ algebra of
$\mathcal{F}$ with respect to $\tilde\Omega$.  Then $\theta$ is
\begin{equation*}
    (\mathcal{F}_{\tilde\Omega}\otimes\mathcal{B}(\mathbb{R}),\mathcal{F}_{\tilde\Omega})-\text{measurable}.
\end{equation*}
\end{lemma}
\begin{proof}
$A^\prime\in \mathcal{F}_{\tilde\Omega}$ if and only if there
exists an $A\in \mathcal{F}$ such that $A^\prime=A\cap\tilde
\Omega$. Hence
\begin{equation*}
    \theta^{-1}(A^\prime)=\theta^{-1}(A)\cap
    (\tilde\Omega\times
    \mathbb{R})\in(\mathcal{F}\cap\mathcal{B}(\mathbb{R}))\cap(\tilde\Omega\times\mathbb{R})
\end{equation*}
by the invariance of $\tilde\Omega$. Let $R$ be the set of
measurable rectangle sets of $\Omega\times \mathbb{R}$. It follows
from Halmos \cite{Hal74} Section 5 Theorem E that
\begin{equation*}
(\mathcal{F}\cap\mathcal{B}(\mathbb{R}))\cap(\tilde\Omega\times\mathbb{R})
=\sigma(R)\cap (\tilde \Omega\times \mathbb{R})=\sigma(R\cap
(\tilde\Omega\times
\mathbb{R}))=\mathcal{F}_{\tilde\Omega}\otimes\mathcal{B}(\mathbb{R}).
\end{equation*}
This completes the proof.            \hfill$\square$
\end{proof}

Let $\tilde {\mathbb{P}}$ be the restriction of  $\mathbb{P}$ to
$\mathcal{F}_{\tilde\Omega}$. In the following we will denote the
new restricted metric dynamical system
$(\tilde\Omega,\mathcal{F}_{\tilde\Omega},\tilde{\mathbb{P}},\theta)$
by $(\Omega,\mathcal{F},{\mathbb{P}},\theta)$.

\medskip

Our considerations are based crucially on the following theorem,
which says that multiplicative ergodic theorem (MET) (i.e., the
existence of Lyapunov exponents) implies nonuniform exponential
dichotomy in infinite dimensional Hilbert spaces.

\begin{theorem} (\textbf{MET implies  nonuniform exponential dichotomy})  \label{t4}  \\
Assume the assumptions of Theorem \ref{t3}. Suppose that
$\lambda_1>\lambda$ and let $\lambda_+$ be the smallest Lyapunov
exponent bigger than $\lambda$ and let $\lambda_-$ be the biggest
Lyapunov exponent smaller than $\lambda$. There exist  a tempered
random variable $K_\lambda^s(\omega)$ and a tempered from below
random variable $K_\lambda^u(\omega)$ such that, for $t\ge 0$,
$\omega\in\Omega$ and $\eps >0$,,
\begin{align*}
    \|U_\lambda^u(t,\omega)\|&\ge
    K_\lambda^u(\omega) \; e^{(\lambda_+-\eps)t}\\
    \|U_\lambda^s(t,\omega)\|&\le
    K_\lambda^s(\omega) \; e^{(\lambda_-+\eps)t}.
\end{align*}
\end{theorem}
\begin{remark}
This nonuniform exponential dichotomy is also called nonuniform
pseudo-hyperbolicity; see Definition \ref{hyperbolic} in the next
section.
\end{remark}
\begin{proof}
We start with $K^u_\lambda$. By Lemma \ref{l2} we can assume that
\begin{equation}\label{eq9}
    \|U^u_\lambda(t,\omega)\|\le H_\omega^\Lambda  e^{\Lambda\, t}\quad
    \text{for }\quad t\ge 0,\omega\in\Omega
\end{equation}
where $\Lambda$ is chosen bigger than $\tilde\Lambda$ (see Lemma
\ref{l2}, (\ref{x1})) and $|\lambda_+|$. Sufficient for the
conclusion of the first part is to show that
\begin{equation*}
   \frac{1}{K^u(\omega)}:=\sup_{t\ge
    0}\frac{e^{(\lambda^u-\eps)t}}{\|U_\lambda^u(t,\omega)\|}
    =\sup_{t\in
    \mathbb{Q}}\frac{e^{(\lambda^u-\eps)t}}{\|U_\lambda^u(t,\omega)\|}
\end{equation*}
is a tempered random variable in $(0,\infty)$. Indeed, to see that
$1/K_\lambda^{+}$ is a random variable we note that  $t\to
\|U_\lambda^u(t,\omega)\|$ is continuous on $(0,\infty)$ by the
finite dimensionality of $E^u(\omega)$. In addition, by
$U(\cdot,\omega)\in C(\mathbb{R}^+,L_s(H))$ the norm
$\|U_\lambda^u(t,\omega)\|$  is bounded away from zero for $t\to
0$. Indeed, We have
\begin{equation*}
    \|U_\lambda^u(t,\omega)\|\ge \|U_\lambda^u(t,\omega)x\|,\qquad
    \|x\|=1,\,x\in E^u(\omega),
\end{equation*}
where the right hand side converges to one as $t\to 0$. Similarly,
we have on  $\Omega$
\begin{equation*}
    \limsup_{t\to\infty}\frac{e^{(\lambda_+-\eps)t}}{\|U_\lambda^u(t,\omega)\|}
    \le
    \limsup_{t\to\infty}\frac{e^{(\lambda_+-\eps)t}}{\|U_\lambda^u(t,\omega)x\|}<\infty,\qquad
    (\|x\|=1)
\end{equation*}
which follows from Theorem \ref{t3}.   According to Lemma \ref{l2}
\begin{equation*}
    e^{-\Lambda t}\le
    \frac{H_\omega^\Lambda }{\|U_\lambda^u(t,\omega)\|}\quad\text{for
    any } t\ge 0.
\end{equation*}
We then see that for $s>0$
\begin{align*}
  \frac{1}{ K^u_\lambda(\theta_s\omega)}e^{-2\Lambda s}&=\sup_{t\ge
    0}\frac{e^{(\lambda_u-\eps)t}}{\|U_\lambda^u(t,\theta_s\omega)\|}e^{-2\Lambda
    s}\\
    &\le \sup_{t\ge
    0}\frac{e^{(\lambda_+-\eps)t}
    e^{(\lambda_+-\eps)s}}{\|U_\lambda^u(t,\theta_s\omega)\|
    \|U_\lambda^u(s,\omega)\|}e^{-\Lambda
    s}H_\omega^\Lambda e^{-(\lambda_+-\eps)s}\\
    &\le
    \sup_{t\ge
    0}\frac{e^{(\lambda_+-\eps)(t+s)}}{
    \|U_\lambda^u(t+s,\omega)\|}e^{-\Lambda
    s}H_\omega^\Lambda e^{-(\lambda_+-\eps)s}\\
    &\le
    \frac{H_\omega^\Lambda }{K_\lambda^u(\omega)}e^{-\Lambda s}e^{-(\lambda_+-\eps)s}
\end{align*}
which goes to zero for $s\to\infty$. Thus the condition
(\ref{eq8}) gives the
first part of the conclusion.\\

\bigskip

 Now we show the existence and temperedness of $K_\lambda^s(\omega)$ on the stable
 space. We first show this for discrete time and then extend it to
 continuous time.
Define

\begin{equation*}
    K_\lambda^s(\omega):=\sup_{t\in
    \mathbb{R}^u}\frac{\|U_\lambda^s(t,\omega)\|}{e^{(\lambda_-+\eps)t}}.
\end{equation*}

We use the Kingman subadditive ergodic theorem (see Theorem A. 1
in Ruelle \cite{Rue82}).

Define $F_n(\om) = \log \| U_\lambda^s(n, \om)\|$. We can check
that $F_n$ satisfies the conditions in the Kingman subadditive
ergodic theorem. Therefore, together with Ruelle's MET
\cite{Rue82}, there exists a
$\{\theta_t\}_{t\in\mathbb{Z}}-$invariant measurable function
$F(\om)$ such that

\begin{equation}\label{kingman1}
 F(\om)=\lim_{n\to \infty} \frac1{n}F_n(\om)
 =  \lim_{n\to \infty} \log \frac{1}{n}\|U_\lambda^s(n, \om)\| \leq
 \lambda_-.
\end{equation}

Now set $G(\om) =\lambda_-$. Then $G(\om) \geq F(\om)$. As a
consequence of Kingman's subadditive ergodic theorem (see
Corollary A.2 in Ruelle \cite{Rue82}), for every $\eps>0$, there
is
 finite-valued random variable $K_{\eps}(\om)$   such that when $n>m$,
\begin{equation}\label{kingman2}
\log\|U_\lambda^s(n-m, \theta_m\om)\| \leq
(n-m)\lambda_-+n\frac{\eps}{2} +K_{\eps}(\om), \;\;  a.s.
\end{equation}

To see the temperedness of $K_\lambda^s$ is is sufficient to show
that

\begin{equation*}
    \lim_{s\to\infty} K_\lambda^s(\theta_s\omega)e^{\Lambda s}=0,
\end{equation*}

 for some  sufficiently large $\Lambda$. But this follows from the
 definition of $K_\lambda^s$ and from

 \begin{align*}
    \|U_\lambda^s(t,\theta_s\omega)\|&\le
    \|U(1+t-[t]-1-[s]+s,\theta_{[s]+[t]}\omega)\|\times\\
    &\times
    \|U_\lambda^s([t]-1,\theta_{1+[s]}\omega)\|\|U(1-s+[s],\theta_s\omega)\|\\
    &\le
    e^{D(\theta_{[s]+[t]+1}\omega)}e^{D(\theta_{[s]+[t]}\omega)}e^{K_\eps(\omega)+\eps([s]+1)+(\lambda_-+\frac{\eps}{2})([t]-1)}
e^{D(\theta_{[s]}\omega)}
\end{align*}
for $t>1$ and similarly for $t\in[0,1]$.   \hfill$\square$

\end{proof}

\medskip


We now are prepared to show that the random partial differential
equation
\begin{equation}\label{eqA1}
    \frac{du}{dt}+A(\theta_t\omega)u=F(\theta_t\omega,u),\qquad
    u(0)=x\in H,
\end{equation}
  via its solution mapping, defines a continuous random
dynamical system. Here the nonlinear term $F$ does not depend on
the gradient of $u$. A similar result for stochastic partial
differential equations is in \cite{Flandoli}.

\begin{theorem} (\textbf{Generation of cocycle})      \label{cocycle} \\
Let
\begin{equation*}
    F:\Omega\times H\to H
\end{equation*}
be a mapping such that $F(\cdot,x)$ is
$(\mathcal{F},\mathcal{B}(H))$ measurable for $x\in H$ and
$F(\omega,\cdot)$ is Lipschitz continuous for $\omega\in \Omega$
with a Lipschitz constant $L(\omega)$ such that
\begin{equation*}
    \int_a^bL(\theta_s\omega)ds<\infty\qquad \text{for
    }-\infty<a<b<\infty.
\end{equation*}
Then (\ref{eqA1}) has a unique (mild) solution on any interval
$[0,T]$ for any $\omega\in\Omega$ which generates a continuous
random dynamical system.
\end{theorem}
\begin{proof}
We consider the Polish space $C_{T,x}:=C([0,T],H)$ of continuous
functions $u$ with values in $H$ and $u(0)=x$. This space is
equipped with the norm
\begin{equation*}
    |||u|||=\sup_{t\in[0,T]}e^{-\Lambda t}\|u(t)\|\quad\text{for
    some  }\Lambda>0.
\end{equation*}
We consider the mapping
\begin{equation}\label{A3}
    \mathcal{T}_{x}(u)[t]:=U(t,\omega)u(0)
    +\int_0^tU(t-s,\theta_s\omega)F(\theta_s\omega,u(s))ds,\quad
    t\in[0,T],\quad u\in C_{T,x}.
\end{equation}
According to Amann \cite{Ama95} Page 46 f. we have that
$\mathcal{T}_{x}(u)\in C_{T,x}$. Due to (\ref{eq6a}) we obtain
\begin{equation*}
    \|U(t-s,\theta_s\omega)\|=\|U_\omega(t,s)\|\le
    C_\omega e^{k_0(t-s)}.
\end{equation*}
We now choose $\Lambda$ sufficiently large such that
\begin{equation}\label{eqA4}
  \max_{t\in[0,T]}  \int_0^te^{-(\Lambda-\mu)(t-s)}CL(\theta_s\omega)ds\le
    \frac{1}{2},\qquad \Lambda>k_0,
\end{equation}
where $\mu=\mu(\omega),\,C=C_\omega$. Indeed $s\to
C\,L(\theta_s\omega)$ is an integrable majorant for the integrand
in the above integral with respect to $\Lambda$. For
$\Lambda\to\infty$ the integrand goes to zero for almost all
$s\in[0,t]$. By the Lebesgue theorem the integrals go to zero for
$\Lambda\to\infty$ and for  any $t$. Note that for fixed $t$ the
integrals are monotone in $\Lambda$ such that, for sufficiently
large $\Lambda$, we have the inequality (\ref{eqA4}) by Dini's theorem.\\
We then have the contraction condition

\begin{align*}
    |||\mathcal{T}_{x}(u)-\mathcal{T}_{x}(v)|||\le
    \max_{t\in[0,T]}\int_0^te^{-(\Lambda-\mu)(t-s)}CL(\theta_s\omega)ds\,|||u-v|||\le\frac{1}{2}|||u-v|||.
\end{align*}
The Banach fixed point theorem gives us a solution of (\ref{eqA1})
which is continuous in $t$ for any $\omega\in\Omega,\,T\ge 0$ and
$\Lambda=\Lambda(\omega,T)$ sufficiently large.\\
The solution of (\ref{eqA1}) depends continuously on $x$ what
follows by the Gronwall Lemma from the Lipschitz continuity of
$F$.\\

The norms $|||\cdot|||$ are equivalent to the standard supremum
norm for every $\Lambda>0$. Hence we can construct the solution of
(\ref{eqA1}) by successive iteration of the operator
$\mathcal{T}_{x}$ starting with the {\em measurable} function
$u^0(t,\omega)\equiv x$. We see that the solution is  a pointwise
limit of measurable functions, hence measurable. Let
$\phi(t,\omega,x)$ be the solution operator for (\ref{eqA1}). The
measurable dependence on $x,\,t,\,\omega$ follows in the same way
as for the linear case, see above. The cocycle property follows by
\begin{align*}
    \phi(t+\tau,\omega,x)&=U(t+\tau,\omega)x+\int_0^{t+\tau}U(t+\tau-s,\theta_s\omega)F(\theta_s\omega,\phi(s,\omega,x))ds\\
    &=U(t,\theta_\tau\omega)(U(\tau,\omega)x+\int_0^\tau
    U(\tau-s,\theta_s\omega)F(\theta_s\omega,\phi(s,\omega,x))ds)\\
    &+
    \int_0^tU(t,\theta_s\theta_\tau\omega)F(\theta_s\theta_\tau\omega,\phi(s,\theta_\tau\omega,\phi(\tau,\omega,x)))ds.
\end{align*}
Hence $\varphi$ is a continuous random dynamical system.
\hfill$\square$
\end{proof}

%
\section{Invariant Manifolds}  \label{IM}

In this section, we consider a general nonlinear random
evolutionary equation in a Hilbert space $H$
\begin{equation}\label{eq3.1}
\frac{du}{dt} + A(\theta_t\omega) u= F(\theta_t\omega, u),
\end{equation}
with the random linear operator $A$ and nonlinear part $F$. We
assume that the linear equation
\begin{equation}\label{eq3.1*}
\frac{du}{dt} + A(\theta_t\omega) u= 0
\end{equation}
generates a linear random dynamical system $U(t, \omega)$ on $H$
for $t\geq 0$. We first introduce a weak hyperbolicity condition
on the linear dynamics.

\begin{definition}  \label{hyperbolic} {\it
$U(t, \omega)$ (or $u=0$) is said to be {\bf nonuniformly
pseudo-hyperbolic}  if there exists a $\theta_t$-invariant set
$\tilde \Omega \subset \Omega$ of full measure such that for each
$\omega\in \tilde \Omega$, the phase space $H$ splits into
\[
H=E^s(\omega)\oplus E^u(\omega)
\]
of closed subspaces satisfying
\begin{itemize}
\item[(i)] This splitting is invariant under $\;U(t, \omega)$:
\begin{align*}
&U(t, \omega) E^s(\omega) \subset E^s(\theta_t\omega) \\
&U(t, \omega) E^u(\omega) \subset E^u(\theta_t\omega) \\
\end{align*}
and $U(t, \omega)|_{E^u(\omega)}$ is an isomorphism from
$E^u(\omega)$ to $E^u(\theta_t \omega)$. \item[(ii)] There are
$\theta$-invariant random variables $\alpha(\omega) >
\beta(\omega)$, and a tempered random variable $K(\omega):\tilde
\Omega \to [1, \infty)$ such that
\begin{align}
\label{hstable}||U(t, \omega)\Pi^s(\omega)||& \leq K(\omega)
e^{\beta(\omega)t}
\quad\hbox{for } t\geq 0 \\
\label{hunstable}||\big(U(-t,
\theta_t\omega)|_{E^u(\theta_t\omega)}\big)^{-1} \Pi^u(\omega)||&
\leq K(\omega) e^{\alpha(\omega)t} \quad\hbox{for } t\leq 0,
\end{align}
where $\Pi^s(\omega)$ and $\Pi^u(\omega)$ are the measurable
projections associated with the splitting. For our special setting
of {\em ergodicity} we can assume that $\alpha>\beta$ are constant
on a $\{\theta_t\}_{t\in\mathbb{R}}$-invariant set of a full
measure.
\end{itemize}}
\end{definition}
Let
\[
U_\lambda^u(t, \omega)=U(t,\omega)|_{E^u(\omega)} \quad
\hbox{and}\quad U_\lambda^s(t, \omega)=U(t,\omega)|_{E^s(\omega)}
\]
where $\lambda$ generates some splitting of $H$, see Section
\ref{ED}. Then, the condition (i) in the above definition implies
that we may extend $U_\lambda^u(t, \omega)$ to be defined for
$t<0$ as
\[
U_\lambda^u(t, \omega)=\big(U_\lambda^u(-t,
\theta_t\omega)\big)^{-1}.
\]
One can easily verify that the cocycle property holds for the
extended system $U_\lambda^u(t,\omega)$ with $t\in \mathbb{R}$.

 \begin{remark}
 As $\omega$ varies,
$\beta(\omega)$ may be arbitrarily small and $K(\omega)$ may be
arbitrarily large. However, along each orbit $\{\theta_t\omega\}$,
$\alpha(\omega)$ and $\beta(\omega)$ are constant and $K(\omega)$
can increase only at a subexponential rate. Thus, the linear
system $U(t, \omega)$ is nonuniformly hyperbolic in the sense of
Pesin. As an example, let $U(t, \omega)$ be  an infinite
dimensional linear random dynamical system satisfying the
conditions of the following multiplicative ergodic theorem. Then,
the nonuniform pseudo-hyperbolicity we introduced here
automatically follows.
 \end{remark}

\vskip0.1in For the remainder of this paper, we assume that

\vskip0.1in \noindent{\bf Hypothesis A:} {\it $U(t, \omega)$ is
nonuniformly pseudo-hyperbolic. }

 \vskip0.1in For the nonlinear term $F(\omega, x)$ we assume
that

\bigskip
\noindent{\bf Hypothesis B:} {\it There is a ball, $
\mathcal{N}(\omega)=B(0, \rho(\omega))=\{u \in H\; | \; ||u|| <
\rho(\omega)\}, $ where $\rho:\Omega \to (0, \infty)$ is tempered
from below and $\rho(\theta_t\omega)$ is locally integrable, such
that $F(\omega, \cdot): \mathcal{N}(\omega) \to H$ is Lipschitz
continuous and satisfies $F(\omega, 0)=0$ and
\[
\vert |F(\omega, u)-F(\omega, v)\vert| \leq \tilde B_1(\omega)
\left(\vert|u\vert|^{\eps}+\vert|v\vert|^{\eps}\right)\vert|u-v\vert|
, \quad \omega \in \Omega, \; u, v \in \mathcal {N}(\omega).
\]
where $\tilde B_1(\omega)$ is a random variable  tempered from
above, $\tilde B_1(\theta_t\omega)$ is locally integrable  in $t$
and $\eps\in (0, 1]$.}

Later we can see that we can {\em extend} such an $F$ to
$\Omega\times H$ such that the assumptions of Theorem
\ref{cocycle} are satisfied.\\

Next, we introduce a modified equation by using a cut-off function
\cite{Arn98}. Let $\sigma(s)$ be a $C^\infty$ function from
$(-\infty, \infty)$ to $[0, 1]$ with
\[
\sigma(s)=1\;\;\hbox{for }\; |s| \leq 1, \quad \sigma(s)=0
\;\;\hbox{for }\; |s| \geq 2,
\]
\[
\sup_{s\in \mathbb{R}} |\sigma'(s)| \leq 2.
\]

Let $\rho:\Omega \to (0, \infty)$ be a random variable tempered
from below such that $\rho(\theta_t\omega)$ is locally integrable
in $t$. We consider a modification of $F(\omega, u)$. Let
\[
F_\rho(\omega,
u)=\sigma\left(\frac{|u|}{\rho(\omega)}\right)F(\omega, u).
\]

An elementary calculation gives
\begin{lemma}\label{cutoff}
\begin{itemize}
\item[(i)] $F_\rho(\omega, u)= F(\omega, u)$, for $|u| \leq
\rho(\omega)$ and $||F_\rho(\omega, u)|| \leq B_0(\omega)$, where
$B(\omega)>0$ is a random variable tempered from above and
$B(\theta_t\omega)$ is locally integrable in $t$ ; \item[(ii)]
there exists a random variable $B_1(\omega)>0$ tempered from
above, $B_1(\theta_t\omega)$ is locally integrable in $t$, such
that
\[
\vert|F_\rho(\omega, u)-F_\rho(\omega, v)\vert| \leq
B_1(\omega)\big(\rho(\omega)\big)^{\eps}||u-v||, \quad \text{for
all } u, v \in H.
\]
\end{itemize}
\end{lemma}

We now consider the following modified equation

\begin{equation}\label{eq3.3}
\frac{du}{dt} + A(\theta_t\omega) u= F_\rho(\theta_t\omega, u).
\end{equation}

Using Lemma \ref{cutoff}, this modified equation has a unique
global solution for each given initial value $u(0)=u_0$, thus
generates a random dynamical system.

\bigskip

We consider the Banach Space for
$\gamma(\omega)=\frac{\alpha(\omega)+\beta(\omega)}{2}$
\[
C_{\gamma}^- = \left\{u|u : \mathbb{R}^- \to E \hbox{ is
continuous and } \sup_{t\leq 0}|| e^{-\gamma(\omega) t } u (t) ||
< \infty \right\}
\]
with the norm $|u|_{\gamma}^-  = \sup_{t\leq 0}
||e^{-\gamma(\omega)t} u(t) ||$. Let $u(t,\omega, u^0)$ denote the
solution of equation (\ref{eq3.3}) and set
$$
M^u(\omega) = \left\{ u^0 | u(t,\omega, u^0) \hbox{ is defined for
all } t \leq 0 \hbox{ and } u (\cdot, \omega,  u^0 ) \in
C_{\gamma}^-\right\}.
$$

Then the set $M^u$ is called local unstable invariant set. If
$M^u(\omega)$ can be defined by a graph of a Lipschitz continuous
function then we call $M^u(\omega)$ Lipschitz pseudo-unstable
manifold for equation (\ref{eq3.3}).

\begin{theorem} (\textbf{Pseudo-unstable manifold theorem}) \label{thm3.1} \\
Assume that Hypotheses A and B hold and choose the tempered radius
$\rho(\omega)$ such that
\begin{equation}\label{radius}
0< \rho(\omega)<
\Big(\frac{\alpha(\omega)-\beta(\omega)}{8K(\omega)B_1(\omega)}\Big)^{\eps}.
\end{equation}Then there exists a Lipschitz pseudo-unstable manifold for equation
(\ref{eq3.3}) which is given by
$$
M^u(\omega) = \{p + h^u(\omega, p ) |p \in E^u(\omega)\}
$$
where $h^u(\omega,  \cdot): E^u(\omega) \to E^s(\omega)$ is
Lipschitz continuous and satisfies $h^u(\omega, 0) = 0$.
\end{theorem}

\begin{remark} When $\alpha(\omega) <0$, the assumption $F(\omega, 0)=0$
can be removed. This corresponds to the inertial manifold in
deterministic case. If $F$ is continuously differentiable  in $u$,
then $h^u$ is continuously differentiable in $u$. Note that $h^u$
and thus the local manifold $M^u$ depend on $\rho$. The proof
below shows the existence of a   unstable manifold for the
truncated equation (\ref{eq3.3}), and as in \cite{CarLanRob01}, it
can be shown that this is indeed a local unstable manifold for the
original equation \eqref{eq3.1}
\end{remark}

\begin{proof}  We use the Lyapunov and Perron approach to show this theorem.
Then $M^u(\omega)$ is nonempty since $u=0\in M^u(\omega)$, and
invariant for the random dynamical system generated by
(\ref{eq3.3}). We will prove that $M^u(\omega)$ is given by the
graph of a Lipschitz function over $E^u(\omega)$.\\

We first claim that for $u(\cdot)\in C_{\gamma}^-(\omega)$ $u(0)
\in M^u(\omega)$ if and only if $u(t)$ satisfies
\begin{align}\label{eq3.5}\begin{split}
u(t) = U_\lambda^u(t, \omega) \xi &+ \int_0^t U_\lambda^u(t-\tau,
\theta_\tau\omega)
 \Pi^u F_\rho(\theta_\tau\omega, u) d\tau\\
&+\int_{-\infty}^t U_\lambda^s(t-\tau, \theta_\tau\omega) \Pi^s
F_\rho(\theta_\tau\omega, u)d\tau,
\end{split}
\end{align}
where $\xi = \Pi^u u(0)$.

To prove this claim, we first let $u(0)=u^0\in M^u(\omega)$. By
using the variation of constants formula, we have
\begin{equation} \label{eq3.6}
\Pi^u u(t) = U_\lambda^u(t, \omega) \Pi^u u^0 + \int_0^t
U_\lambda^u(t-\tau, \theta_\tau\omega) \Pi^u F_\rho(\theta_\tau
\omega,u) d\tau,
\end{equation} and for $t_0\leq t$
\begin{align}\begin{split}\label{eq3.7}
\Pi^su(t) &=U_\lambda^s(t-t_0, \theta_{t_0}\omega)
\Pi^su(t_0)+\int_{t_0}^t U_\lambda^s(t-\tau, \theta_\tau\omega)
\Pi^sF_\rho(\theta_\tau \omega,u) d\tau.\end{split}
\end{align}
Since $u \in C_{\gamma}^-$, we have, for $t_0 < t , t_0<0$, that
\begin{align*}
&||U_\lambda^s(t-t_0, \theta_{t_0}\omega)\Pi^s u (t_0)||
\leq K(\theta_{t_0}\omega) e^{\beta(\omega)(t- t_0)} e^{\gamma(\omega)t_0} |u|_{\gamma}^-\\
&\leq e^{\beta(\omega) t}\Big(K(\theta_{t_0}\omega) e^{(\gamma
(\omega)-\beta(\omega))t_0}\Big) |u|_{\gamma} \to 0 \quad \hbox{
as } \quad t_0 \to - \infty,
\end{align*}
where we used the facts that $\beta(\omega)< \gamma(\omega)$ and
$K(\omega)$ is tempered from above. Taking the limit $t_0 \to -
\infty$ in (\ref{eq3.7}),
\begin{equation}\label{eq3.8}
\Pi^s u (t) = \int_{-\infty}^t U_\lambda^s(t-\tau,
\theta_\tau\omega) \Pi^sF_\rho(\theta_\tau \omega,u) d\tau.
\end{equation}

Combining (\ref{eq3.6}) and (\ref{eq3.8}), we obtain
(\ref{eq3.5}). The converse follows from a direct computation.

Let ${\mathcal J}^u(u, p, \omega)$ be the right hand side of
equality (\ref{eq3.5}).  Using (\ref{hstable}), (\ref{hunstable}),
Lemma \ref{cutoff}, and (\ref{radius}),  we have for $u, \bar u
\in C_{\gamma}^-$
\begin{align*}
&|{\mathcal J}^u (u, p, \omega) - {\mathcal J}^u(\bar u, p,
\omega)|_{\gamma}^-\\
&\quad\leq \sup_{t\leq 0}\Big\{ \int_t^0
e^{\big(\alpha(\omega)-\gamma(\omega)\big)(t-\tau)}
K(\theta_\tau\omega)B_1(\theta_\tau\omega)\rho^\eps(\theta_\tau\omega)
d\tau\\
&\qquad\qquad + \int_{-\infty}^t
e^{-\big(\gamma(\omega)-\beta(\omega)\big)(t-\tau)}
K(\theta_\tau\omega)B_1(\theta_\tau\omega)\rho^\eps(\theta_\tau\omega)
d\tau   \Big\}|u-\bar u |_{\gamma}^-\\
&\quad \leq \frac1{2}  |u-\bar u |_{\gamma}^-
\end{align*}
and
$$
|{\mathcal J}^u(u, p, \omega) - {\mathcal J}^u (u, \bar p,
\omega)|_{\gamma}^- \leq K(\omega)||p - \bar p||.
$$
Using the uniform contraction mapping principle, we have that for
each $p \in E^u(\omega)$ ${\mathcal J}^u$ has a fixed point, thus
equation (\ref{eq3.5}) has a unique solution $u(\cdot, p,
\omega)\in C_{\gamma}^-$ which is Lipschitz continuous in $p$ and
satisfies
\begin{align*}
|u(\cdot, p,  \omega)- u(\cdot, \bar p,  \omega)|_{\gamma}^-\leq
2K(\omega) || p - \bar p||.
\end{align*}

Let
\begin{align*}
h^u (\omega, p) = \Pi^s u(0, p,\omega) = \int_{-\infty}^0
U_\lambda^s(-\tau, \theta_\tau\omega) \Pi^s
F_\rho(\theta_\tau\omega, u(\tau,p, \omega)) d\tau.
\end{align*} Then $h^u(\omega, 0) = 0$ and $h^u(\omega, \cdot)$ is Lipschitz
continuous.

By the definition of $h^u$ and the fact that $u^0\in M^u(\omega)$
if and only if (\ref{eq3.5}) has a unique solution $u(\cdot)$ in
$C^-_{\gamma}$ with $u(0)=u^0=p+h^u(\omega, p)$ for some $p\in
E^u(\omega)$, it follows that

$$ M^u(\omega) = \{ p + h^u (\omega, p)
|p\in E^u(\omega) \}.
$$
This completes the proof of the pseudo-unstable manifold theorem.
\hfill$\square$
\end{proof}

\begin{theorem} (\textbf{Pseudo-stable manifold theorem}) \label{thm3.2} \\
Assume that Hypotheses A and B hold and choose the same tempered
radius as in Theorem \ref{thm3.1}. Then there exists a Lipschitz
pseudo-stable manifold for equation (\ref{eq3.3}) which is given
by
$$
M^s(\omega) = \{q + h^s(\omega, q ) |q \in E^s(\omega)\}
$$
where $h^s(\omega,  \cdot): E^s(\omega) \to E^u(\omega)$ is
Lipschitz continuous and satisfies $h^u(\omega, 0) = 0$
\end{theorem}

\begin{remark}
Restricting $M^u(\omega)$ and $M^s(\omega)$ to a random ball
$\mathcal{N}(\omega)$ with center zero and a random radius
tempered from below gives local random pseudo-unstable and
pseudo-stable manifolds for equation (\ref{eq3.1}), respectively,
see Lu and Schmalfu{\ss} \cite{r-9}.
\end{remark}

\begin{proof}  When $H$ is a finite dimensional space, one can simply reverse the time to get
the pseudo stable manifold by using the pseudo-unstable manifold
theorem. For an infinite dimensional space $H$, since the random
dynamical systems are generally defined only for $t\geq 0$, the
pseudo-unstable manifold theorem cannot be applied here as for the
finite dimensional systems. Define the following Banach space for
$\gamma(\omega)=\frac{\alpha(\omega)+\beta(\omega)}{2}$
$$
C_{\gamma}^u = \left\{ u | u : \mathbb{R}^u \to E \hbox{ is
continuous and } \sup_{t\geq 0} || e^{\gamma t} u (t) || < \infty
\right\}
$$
with the norm $|u|_{\gamma}^u = \sup_{t \geq 0} || e^{\gamma t} u (t) ||$. \\

Let
\[
M^s(\omega)=\big\{u^0\;:\; u(\cdot, \omega,  u^0) \in
C_{\gamma}^u\big\}.
\]It is easy to see that $M^s(\omega)$ is nonempty and
invariant for the random dynamical system generated by equation
(\ref{eq3.3}). We will show that $M^s(\omega)$ is the graph of a
Lipschitz function over $E^s(\omega)$. First, a similar
computation as in the proof of Theorem \ref{thm3.1} gives that,
for $u(\cdot)\in C_{\gamma}^+$, $u(0) \in M^s(\omega)$ if and only
if $u(t)$ satisfies
\begin{align}
\begin{split}\label{eq3.9} u(t) &= U_\lambda^s(t, \omega)q +
\int_0^t U_\lambda^s(t-\tau, \theta_\tau\omega) \Pi^s
F_\rho(\theta_\tau\omega, u(\tau)) \\
&\qquad\quad\qquad + \int_\infty^t U_\lambda^u(t-\tau,
\theta_\tau\omega) \Pi^u F_\rho(\theta_\tau\omega, u(\tau)) d\tau,
\end{split}
\end{align} where $q
= \Pi^s u (0)$.

We will show that for each $q \in E^s(\omega)$, equation
(\ref{eq3.9}) has a unique solution in $C_{\gamma}^+$. To see
this, let ${\mathcal J}^s(u, q, \omega)$ be the right hand side of
(\ref{eq3.9}). A simple calculation gives that ${\mathcal J}^s$ is
well-defined from $C_{\gamma}^+$ to itself for each fixed
$\omega\in\Omega$ and $q\in E^s(\omega)$. For any $u,\bar u \in
C_{\gamma}^+$, using (\ref{hstable}), (\ref{hunstable}), Lemma
\ref{cutoff}, and (\ref{radius}),  we have
\begin{equation}\label{eq3.11}
|{\mathcal J}^s (u,q,\omega) - {\mathcal J}^s(\bar u, q,
\omega)|_{\gamma}^+ \leq \frac12 | u - \bar u |_{\gamma}^+
\end{equation}
and
\[
|{\mathcal J}^s (u,q,\omega) - {\mathcal J}^s(u, \bar q,
\omega)|_{\gamma}^+ \leq K(\omega) \vert| q - \bar q \vert|.
\]

Using the uniform contraction principle, we have that for each
$\omega\in \Omega$ and $q \in E^s(\omega)$ equation (\ref{eq3.9})
has a unique solution $u(\cdot, q, \omega) \in C_{\gamma}^+$ which
is Lipschitz continuous in $q$ and satisfies
\begin{equation}\label{eq4.6}
|u(\cdot, q, \omega) - w(\cdot, \bar q, \omega) |^+_{\gamma} \leq
2K(\omega)|| q - \bar q ||.
\end{equation}

Let $h^s(\omega, q)=\Pi^su(0, q, \omega)$. Then
\begin{align*}
h^s(q,\omega) = \int_\infty^0 U_\lambda^u(-\tau,
\theta_\tau\omega) \Pi^u F_\rho(\theta_\tau \omega, u(\tau,
q,\omega) d\tau,
\end{align*}
$h^s(\omega, 0)=0$, $h^s(\omega, q)$ is Lipschitz in $q$. Using
(\ref{eq3.9}) and the definitions of $M^s(\omega)$ and $h^s$, we
\[
M^s (\omega) = \{ q+ h^s(\omega, q) \; :\; q \in E^s(\omega) \}.
\]
This  proves the pseudo-stable manifold theorem. \hfill$\square$
\end{proof}

\section{An application} \label{appl}

In this section we will illustrate the above random invariant
manifold theory by applying it to an example of stochastic partial
differential equations.

Let $H$ be a separable Hilbert spaces with scalar product
$(\cdot,\cdot)$ and norm $|\cdot|$. Consider an (unbounded)
operator $A:D(A)=:H_1\rightarrow H$ which is supposed that $-A$ is
the generator of a analytic $C_0$--semigroup $\{S_A(t)\}_{t\geq0}$
on $H$, being $S_A(t)$ compact for all $t>0,$ and
so that $-A$ possesses infinitely many eigenvalues%
\[
\mu_{1}\geq\cdots\geq\mu_{j}\geq\mu_{j+1}\geq\mu_{j+2}\geq\cdots\text{
\ \ (with \ }\mu_{j}\rightarrow-\infty\text{ \ as \ }j\rightarrow
\infty\text{)},%
\]
with the associated eigenvalues $\{e_{j}\}_{j\geq1}$ forming a
complete orthonormal basis of $H.$

For instance, we can consider as operator $A$ the one given in
\eqref{opA} which satisfies the homogeneous Dirichlet boundary
conditions (\ref{opA2}), assuming that is symmetric and has a
compact resolvent. Then the above assumptions are satisfied with
$H=L^2(\mathcal{O})$.

On the other hand, assume that $f$ is a Lipschitz continuous
operator from $H$
to $H$, i.e.%
\[
\Vert f(u_{1})-f(u_{2})\Vert\leq L_{f}\Vert
u_{1}-u_{2}\Vert,\qquad\text{for all \ }u_{1},\,u_{2}\in H,
\]
$w_{1},\cdots,w_{N}\,$ are one-dimensional mutually independent
standard Wiener processes over the same probability space, and
$D_{i}\in\mathcal{L}(H)$ for $i=1,\cdots,N.$ Then, we consider the
following semilinear stochastic partial differential equation with
multiplicative Stratonovich linear noise
\begin{equation}
dX+AX\,dt=f(X)dt+\sum_{i=1}^{N}D_{i}X\circ dw_{i}. \label{seq41}%
\end{equation}
The operators $D_{i}$ generate $C_{0}$-groups which we will denote
by
$S_{D_{i}}$. If, in addition, we suppose the operators $A,\,D_{1}%
,\,\cdots,\,D_{N}$ \ mutually commute (what implies that these
groups and the semigroup $S_A(t)$ generated by $A$ are also
mutually commuting), then this stochastic equation will generate a
random dynamical system by performing a suitable transformation
(see Lemma \ref{lm1}). \newline

We consider the one-dimensional stochastic differential equation
\begin{equation}
dz=-\nu\,z\,dt+dw(t) \label{eqx3}%
\end{equation}
for some $\nu>0$. This equation has a random fixed point in the
sense of random dynamical systems generating a stationary solution
known as the stationary Ornstein-Uhlenbeck process.

\begin{lemma}
\label{lx1} (\cite{CaKloSchm04}) Let $\nu$ be a positive number
and consider the probability space as in Section \ref{s2}. There
exists a $\{\theta_{t}\}_{t\in\mathbb{R}}$-invariant subset
$\bar{\Omega}\in \mathcal{F}$ of
$\Omega=C_{0}(\mathbb{R},\mathbb{R})$ of full measure such that
\begin{equation}
 \lim_{t\rightarrow\pm\infty}\frac{|\omega(t)|}{t}=0,
\end{equation}
and, for such $\omega$, the random variable given by
\[
z^{\ast}(\omega):=-\nu\int_{-\infty}^{0}e^{\nu\tau}\omega(\tau)d\tau
\]
is well defined. Moreover, for $\omega\in\bar{\Omega}$, the
mapping
\[%
\begin{split}
(t,\omega)\rightarrow z^{\ast}(\theta_{t}\omega)  &
=-\nu\int_{-\infty
}^{0}e^{\nu\tau}\theta_{t}\omega(\tau)d\tau\\
& =-\nu\int_{-\infty}^{0}e^{\nu\tau}\omega(t+\tau)d\tau+\omega(t)
\end{split}
\]
is a stationary solution of (\ref{eqx3}) with continuous
trajectories. In addition, for $\omega\in\bar{\Omega}$
\[%
\begin{split}
&  \lim_{t\rightarrow\pm\infty}\frac{|z^{\ast}(\theta_{t}\omega)|}%
{|t|}=0,\qquad\lim_{t\rightarrow\pm\infty}\frac{1}{t}\int_{0}^{t}z^{\ast
}(\theta_{\tau}\omega)d\tau=0,\\
&
\lim_{t\rightarrow\pm\infty}\frac{1}{t}\int_{0}^{t}|z^{\ast}(\theta_{\tau
}\omega)|d\tau=\mathbb{E}|z^{\ast}|<\infty.
\end{split}
\]

\end{lemma}

Let $\nu_{1},\cdots,\nu_{N}$ be a set of positive numbers. For any
pair $\nu_{j},w_{j}$ we have a stationary Ornstein-Uhlenbeck
process generated by a random variable $z_{j}^{\ast}(\omega)$ on
$\bar\Omega_{j}$ with properties formulated in Lemma \ref{lx1}
defined on the metric dynamical system
$(\bar\Omega_{j},\mathcal{F}_{j},\mathbb{P}_{j},\theta)$. We set
\begin{equation}
\label{eqx31}(\Omega,\mathcal{F},\mathbb{P},\theta),
\end{equation}
where
\[
\Omega=\bar\Omega_{1}\times\cdots\times\bar\Omega_{N},\quad\mathcal{F}%
=\bigotimes_{i=1}^{N}\mathcal{F}_{i},\quad\mathbb{P}=\mathbb{P}_{1}%
\times\mathbb{P}_{2}\times\cdots\times\mathbb{P}_{N},
\]
and $\theta$ is the flow of Wiener shifts. \newline

To find random fixed points for (\ref{seq41}) we will transform
this equation into an evolution equation with random coefficients
but without white noise. Let
\[
T(\omega):=S_{D_{1}}(z_{1}^{\ast}(\omega))\circ\cdots\circ S_{D_{N}}%
(z_{N}^{\ast}(\omega))
\]
be a family of random linear homeomorphisms on $H$. The inverse
operator is well defined by
\[
T^{-1}(\omega):=S_{D_{N}}(-z_{N}^{\ast}(\omega))\circ\cdots\circ S_{D_{1}%
}(-z_{1}^{\ast}(\omega))
\]
Because of the estimate
\[
\Vert T^{-1}(\omega)\Vert\leq e^{\Vert D_{1}\Vert|z_{1}^{\ast}(\omega)|}%
\cdot\ldots\cdot e^{\Vert D_{N}\Vert|z_{N}^{\ast}(\omega)|}%
\]
and the properties of the Ornstein-Uhlenbeck processes, it follows
that $\|T(\theta_t\omega)\|$, $\Vert
T^{-1}(\theta_{t}\omega)\Vert$ has sub-exponential growth as
$t\rightarrow\pm\infty$ for any $\omega\in\Omega$. Hence $\|T\|$,
$\Vert T^{-1}\Vert$ are tempered. On the other hand, since
$z_{j}^{\ast },\,j=1,\cdots,N$ are independent Gaussian random
variables we have that
\[
\prod_{j=1}^{N}\mathbb{E}(\Vert S_{D_{j}}(-z_{j}^{\ast})\Vert\Vert S_{D_{j}%
}(z_{j}^{\ast})\Vert)<\infty.
\]
Hence by the ergodic theorem we still have a $\{\theta_{t}\}_{t\in\mathbb{R}}%
$-invariant set $\bar{\Omega}\in\mathcal{F}$ of full measure such
that
\[%
\begin{split}
\lim_{t\rightarrow\pm\infty}\frac{1}{t}\int_{0}^{t}\Vert T(\theta_{\tau}%
\omega)\Vert\Vert T^{-1}(\theta_{\tau}\omega)\Vert d\tau &
=\mathbb{E}\Vert
T\Vert\Vert T^{-1}\Vert\\
&  \leq\prod_{j=1}^{N}\mathbb{E}(\Vert
S_{D_{j}}(-z_{j}^{\ast})\Vert\Vert S_{D_{j}}(z_{j}^{\ast})\Vert).
\end{split}
\]
We can change our metric dynamical system with respect to
$\bar{\Omega}$. However the new metric dynamical system will be
denoted by the old symbols
$(\Omega,\mathcal{F},\mathbb{P},\theta)$.\newline\newline

We formulate an evolution equation with random coefficients but
without white noise
\begin{equation}
\label{eqy1}\frac{d\psi}{dt}+\left(
A-\sum_{j=1}^{N}\nu_{j}z_{j}^{\ast }(\theta_{t}\omega)D_{j}\right)
\psi=T^{-1}(\theta_{t}\omega)f(T(\theta _{t}\omega)\psi),
\end{equation}
and initial condition $\psi(0)=x\in H$.

\begin{lemma}\label{lm2}
Suppose that $A,\,D_{1},\cdots,D_{N}$ satisfy the preceding
assumptions. Then
\newline i) the random evolution equation (\ref{eqy1}) possesses a unique
solution, and this solution generates a random dynamical system.
\newline ii) if $\psi$ is the random dynamical system in i),
\begin{equation}
\phi(t,\omega,x)=T(\theta_{t}\omega)\psi(t,\omega,T^{-1}(\omega)x)
\label{eqy3}%
\end{equation}
is another random dynamical system for which the process
\[
(\omega,t)\rightarrow\phi(t,\omega,x)
\]
solves (\ref{seq41}) for any initial condition $x\in H$.
\end{lemma}

\bigskip

From now on, we work with the random partial differential equation
(\ref{eqy1}) which has been obtained (by conjugation) from our
original stochastic PDE. To set our problem in the framework
previously developed, we
denote%
\[
C(\omega)=\sum_{j=1}^{N}\nu_{j}z_{j}^{\ast}(\omega)D_{j},\text{
\ \ \ }A(\omega)=A-C(\omega),\text{ \ \ \ }F(\omega,\cdot)=T^{-1}%
(\omega)f(T(\omega)\cdot).
\]
Note that \ $F(\omega,\cdot)$ is also Lipschitz continuous. The
Lipschitz constant $L$ is locally integrable in the sense of
Theorem \ref{cocycle}.\\

In order to prove the existence of invariant (stable and unstable)
manifolds, we need to check that assumptions in Theorems
\ref{thm3.1} and \ref{thm3.2} are fulfilled. To this end, we first
need to work with the linear part of the RPDE and prove that the
solution operator $U(t,\omega)$ generated by $A(\theta_{t}\omega)$
is nonuniformly pseudo-hyperbolic, what is immediately implied by
the MET (Theorem \ref{t3}). So, it is sufficient to prove the
integrability condition (\ref{eq21}) in that theorem.

Indeed, we define $U(t,\omega)$ by%
\[
U(t,\omega)=S_A(t)\exp\left\{
\int_{0}^{t}C(\theta_{s}\omega)\text{ d}s\right\}  .
\]
Then, defining, for fixed $\omega\in\Omega,$ $u\in H,$ the function%
\[
v(t)=U(t,\omega)u,
\]
and thanks to the commutativity properties of the operators, we
have,
\begin{align*}
\frac{{d}}{{d}t}v(t)  &  =-AS_A(t)\exp\left\{ \int_{0}^{t}C(\theta
_{s}\omega)\text{ d}s\right\} u\\
&+S_A(t)\exp\left\{ \int_{0}^{t}C(\theta
_{s}\omega)\text{ d}s\right\}  C(\theta_{t}\omega)u\\
&  =-AU(t,\omega)u+C(\theta_{t}\omega)S_A(t)\exp\left\{
\int_{0}^{t}C(\theta
_{s}\omega)\text{ d}s\right\}  u\\
&  =-A(\theta_{t}\omega)v(t),
\end{align*}
therefore, $U(t,\omega)$ is the fundamental solution for the linear problem%
\[
\frac{{d}}{{d}t}v(t)+A(\theta_{t}\omega)v(t)=0.
\]

Observe that the compactness of $S_A(t)$ and the commutativity
property implies that $U(t,\omega)$ is also compact.

Let us now prove that assumption (\ref{eq21})is satisfied. Indeed,
take
$t_{1},t_{2}\in\lbrack0,1],$ then%
\begin{align*}
||U(t_{1},\theta_{t_{2}}\omega)|| & \leq||S_A(t_{1})||\left\Vert
\exp\left\{ \int_{0}^{t_{1}}C(\theta_{t_{2}+s}\omega)\text{
d}s\right\}
\right\Vert \\
&  \leq||S_A(t_{1})||\exp\left\{  \int_{0}^{t_{1}}\left\Vert
C(\theta_{t_{2}+s}\omega)\right\Vert \text{ d}s\right\}  ,
\end{align*}
and%
\begin{align*}
\log^{+}||U(t_{1},\theta_{t_{2}}\omega)|| &  \leq\log^{+}%
||S_A(t_{1})||+\int_{0}^{t_{1}}\left\Vert C(\theta_{t_{2}+s}%
\omega)\right\Vert \text{ d}s\\
&  \leq\mu_{1}t_{1}+\int_{t_{2}}^{t_{1}+t_{2}}\left\Vert C(\theta_{s}%
\omega)\right\Vert \text{ d}s\\
&  \leq|\mu_{1}|+\int_{0}^{2}\left\Vert
C(\theta_{s}\omega)\right\Vert \text{ d}s.
\end{align*}
Therefore,%
\[
E\left(  \sup_{t_{1},t_{2}\in\lbrack0,1]}\log^{+}||U(t_{1},\theta_{t_{2}%
}\omega)||\right)  \leq|\mu_{1}|+\int_{0}^{2}E\left\Vert C(\theta
_{s}\omega)\right\Vert \text{ d}s<+\infty
\]
thanks to the properties of the Ornstein-Uhlenbeck processes.

\bigskip
Hence we can apply Theorem \ref{t4} to find the existence of
tempered random variables $K_\lambda^s,\,1/(K_\lambda^u)$ such
that $K=K_\lambda^s+1/(K_\lambda^u)$. For the sake of
completeness, we will explicitly determine the Lyapunov exponents
of $U$ as well as $\alpha(\omega),\beta (\omega),\,K(\omega)$ in
(\ref{hstable})--(\ref{hunstable}). First, we will prove that
\[
\lim_{t\rightarrow+\infty}\frac{1}{t}\log||U(t,\omega)u||\neq
-\infty,\text{ \ \ \ for all \ }u\in H.
\]
This fact implies that there exists infinitely many Lyapunov
exponents.

Choose an eigenvector $e_{j}$ of the operator $A$ associated to
the eigenvalue $\mu_{j.}$ Then,
\begin{align*}
\lim_{t\rightarrow+\infty}\frac{1}{t}\log||U(t,\omega)e_{j}|| &
=\lim_{t\rightarrow+\infty}\frac{1}{t}\log\text{e}^{\mu_{j}t}\left\Vert
\exp\left\{  \int_{0}^{t}C(\theta_{s}\omega)\text{ d}s\right\}  e_{j}%
\right\Vert \\
&  =\mu_{j}+\lim_{t\rightarrow+\infty}\frac{1}{t}\log\left\Vert
\exp\left\{
\int_{0}^{t}C(\theta_{s}\omega)\text{ d}s\right\}  e_{j}\right\Vert \\
&  =\mu_{j,}%
\end{align*}

since, by the ergodic theorem, and the following inequalities

\begin{align*}
    e^{-\big\|\int_0^t C(\theta_s)ds\big\|}&\le\big\|e^{\int_0^t
    C(\theta_s)ds}\big\|\le e^{\big\|\int_0^t
    C(\theta_s)ds\big\|}\\
    &\frac{1}{\big\|e^{-\int_0^t
    C(\theta_s)ds}\big\|}\le \big\|e^{\int_0^t
    C(\theta_s)ds} e_j\big\|
\end{align*}

it easily follows that

\begin{equation*}
    \lim_{t\to\infty}\frac{1}{t}\log\big\|\exp\big(\int_0^tC(\theta_s\omega)ds\big)\big\|=0,
\end{equation*}

and, as a consequence,

\begin{align*}
    0=&\lim_{t\to\infty}\frac{1}{t}\log\big\|\exp\big(-\int_0^tC(\theta_s\omega)ds\big)\big\|
    \le
    \lim_{t\to\infty}\frac{1}{t}\log\big\|\exp\big(\int_0^tC(\theta_s\omega)ds\big)e_j\big\|\\
    \le&
    \lim_{t\to\infty}\frac{1}{t}\log\big\|\exp\big(\int_0^tC(\theta_s\omega)ds\big)\big\|=0,
\end{align*}

so that the Lyapunov exponents $\lambda_j$ for the random
dynamical system $U$ are equal to the eigenvalues $\mu_j$. As for
the associated space $V_j$ it is easy to check that

\begin{equation*}
    V_j=\bigoplus_{i=j}^\infty F_i
\end{equation*}

where $F_i$ are the eigenspaces associated to  $\mu_j$.

\bigskip Let us now determine $\alpha(\omega),\beta(\omega)$ \ and
\ $k(\omega)$ satisfying relations
(\ref{hstable})--(\ref{hunstable}). To this end, let us
denote by $\mu_{s}$ and $\mu_{u}$ the consecutive eigenvalues which satisfy%
\[
\mu_{s}=\mu_{j+1}<0<\mu_{u}=\mu_{j.}%
\]
Then,%
\[
\left\Vert \exp\left\{  \int_{0}^{t}C(\theta_{s}\omega)\text{
d}s\right\}
S_A(t)\Pi^s\right\Vert \leq\text{e}^{\mu_{s}t}\exp\left\Vert \int_{0}%
^{t}C(\theta_{s}\omega)\text{ d}s\right\Vert
\]
and we observe that%
\begin{align*}
\exp\left\Vert \int_{0}^{t}C(\theta_{s}\omega)\text{
d}s\right\Vert  & =\exp\left\Vert
\sum_{j=1}^{N}\nu_{j}\int_{0}^{t}z_{j}^{\ast}(\theta
_{s}\omega)D_{j}\text{ d}s\right\Vert \\
& \leq\exp\left\{  \sum_{j=1}^{N}|\nu_{j}|\left\Vert
D_{j}\right\Vert \left\vert
\int_{0}^{t}z_{j}^{\ast}(\theta_{s}\omega)\text{ d}s\right\vert
\right\}  .
\end{align*}
As%
\[
\lim_{t\rightarrow+\infty}\frac{1}{t}\int_{0}^{t}z_{j}^{\ast}(\theta_{s}%
\omega)\text{ d}s=0,
\]
then for a given $\varepsilon>0,$ there exists $T(\varepsilon)>0$
such that
\[
\left\vert \int_{0}^{t}z_{j}^{\ast}(\theta_{s}\omega)\text{
d}s\right\vert \leq\frac{\varepsilon}{\delta}t,\text{ \ \ for all
\ }t\geq T(\varepsilon ),\text{ \ and all \ \
}j=1,2,\cdot\cdot\cdot,N,
\]
where $\delta=\sum_{j=1}^{N}|\nu_{j}|\left\Vert D_{j}\right\Vert
.$

Thus,%
\[
\exp\left\{  \sum_{j=1}^{N}|\nu_{j}|\left\Vert D_{j}\right\Vert
\left\vert \int_{0}^{t}z_{j}^{\ast}(\theta_{s}\omega)\text{
d}s\right\vert \right\} \leq\text{e}^{\varepsilon t},\text{ \ \
for \ all \ }t\geq T(\varepsilon).
\]
On the other hand, for $t\in\lbrack0,T(\varepsilon))\,$\ we have%
\[
\exp\left\{  \sum_{j=1}^{N}|\nu_{j}|\left\Vert D_{j}\right\Vert
\left\vert \int_{0}^{t}z_{j}^{\ast}(\theta_{s}\omega)\text{
d}s\right\vert \right\} \leq\exp\left\{
\sum_{j=1}^{N}|\nu_{j}|\left\Vert D_{j}\right\Vert
\max_{r\in\lbrack0,T(\varepsilon)]}\left\vert \int_{0}^{r}z_{j}^{\ast}%
(\theta_{s}\omega)\text{ d}s\right\vert \right\}  ,
\]
whence%
\[
\exp\left\{  \sum_{j=1}^{N}|\nu_{j}|\left\Vert D_{j}\right\Vert
\left\vert \int_{0}^{t}z_{j}^{\ast}(\theta_{s}\omega)\text{
d}s\right\vert \right\} \leq\exp\left\{
\sum_{j=1}^{N}|\nu_{j}|\left\Vert D_{j}\right\Vert
\max_{r\in\lbrack0,T(\varepsilon)]}\left\vert \int_{0}^{r}z_{j}^{\ast}%
(\theta_{s}\omega)\text{ d}s\right\vert \right\}
\text{e}^{\varepsilon t},
\]
for all $t\geq0$, and, finally,%
\[
\left\Vert \exp\left\{  \int_{0}^{t}C(\theta_{s}\omega)\text{
d}s\right\} S_A(t)\Pi^s\right\Vert \leq\text{e}^{\left(
\mu_{s}+\varepsilon\right) t}K(\omega),
\]
where%
\[
K(\omega)=\prod\limits_{j=1}^{N}\underset{=K_{j}(\omega)}{\underbrace
{\exp\left\{  |\nu_{j}|\left\Vert D_{j}\right\Vert \max_{r\in
\lbrack0,T(\varepsilon)]}\left\vert \int_{0}^{r}z_{j}^{\ast}(\theta_{s}%
\omega)\text{ d}s\right\vert \right\}  }}.
\]
It is clear that $\beta(\omega)=\mu_{s}+\varepsilon$, and we need
to prove that $K(\omega)$ is tempered. For this, it is enough to
prove that each
$K_{j}(\omega)$ is tempered. Indeed, observe that%
\begin{align*}
0  & \leq\frac{1}{t}\log^{+}K_{j}(\theta_{t}\omega)\\
& =\frac{1}{t}|\nu_{j}|\left\Vert D_{j}\right\Vert \max_{r\in
\lbrack0,T(\varepsilon)]}\left\vert
\int_{0}^{r}z_{j}^{\ast}(\theta
_{s+t}\omega)\text{ d}s\right\vert \\
& \leq|\nu_{j}|\left\Vert D_{j}\right\Vert \frac{1}{t}\max_{r\in
\lbrack0,T(\varepsilon)]}\int_{0}^{r}\left\vert
z_{j}^{\ast}(\theta
_{s+t}\omega)\right\vert \text{ d}s\\
& \leq|\nu_{j}|\left\Vert D_{j}\right\Vert \frac{1}{t}\int_{0}%
^{T(\varepsilon)}\left\vert
z_{j}^{\ast}(\theta_{s+t}\omega)\right\vert \text{
d}s\\
& \leq|\nu_{j}|\left\Vert D_{j}\right\Vert \frac{1}{t}\int_{t}%
^{t+T(\varepsilon)}\left\vert
z_{j}^{\ast}(\theta_{s}\omega)\right\vert \text{
d}s\\
& \leq|\nu_{j}|\left\Vert D_{j}\right\Vert \left(
\underset{\rightarrow
1}{\underbrace{\frac{t+T(\varepsilon)}{t}}}\cdot\underset{\rightarrow
E|z_{j}^{\ast}|}{\underbrace{\frac{1}{t+T(\varepsilon)}\int_{0}%
^{t+T(\varepsilon)}\left\vert
z_{j}^{\ast}(\theta_{s}\omega)\right\vert \text{
d}s}}-\underset{\rightarrow
E|z_{j}^{\ast}|}{\underbrace{\frac{1}{t}\int
_{0}^{t}\left\vert z_{j}^{\ast}(\theta_{s}\omega)\right\vert \text{ d}s}%
}\right)  \rightarrow0\,
\end{align*}
as $t\to\infty$. So, $K(\omega)$ is tempered. A similar analysis
can be carried out to determine that
$\alpha(\omega)=\mu_{u}-\varepsilon.$

Therefore, as the nonlinear term $F$ is globally Lipschitz we can
take $B_{1}=L_{f}$ and assumptions in theorems \ref{thm3.1} and
\ref{thm3.2} are fulfilled. We thus have existence of
pseudo-unstable and pseudo-stable manifolds.

\bigskip

\textbf{Acknowledgements.} This work was started in the summer
2003 when the authors participated in a Research in Teams Program,
  supported by the
Banff International Research Station (Banff, Alberta, Canada).

\end{document}